\RequirePackage[l2tabu, orthodox]{nag}
\documentclass[a4paper]{amsart}
\usepackage{amscd}

\usepackage{amsmath, amsfonts,amssymb,amsthm,mathrsfs,xypic}
\usepackage{graphicx}
\DeclareFontFamily{U}{mathx}{\hyphenchar\font45}
\DeclareFontShape{U}{mathx}{m}{n}{
      <5> <6> <7> <8> <9> <10>
      <10.95> <12> <14.4> <17.28> <20.74> <24.88>
      mathx10
      }{}
\DeclareSymbolFont{mathx}{U}{mathx}{m}{n}
\DeclareFontSubstitution{U}{mathx}{m}{n}
\DeclareMathAccent{\widecheck}{0}{mathx}{"71}
\DeclareMathAccent{\wideparen}{0}{mathx}{"75}

\xyoption{curve}

\usepackage[all]{xy}
\usepackage{microtype}

\numberwithin{equation}{section}

\theoremstyle{plain}
\newtheorem{theorem}[equation]{Theorem}
\newtheorem{proposition}[equation]{Proposition}
\newtheorem{lemma}[equation]{Lemma}
\newtheorem{corollary}[equation]{Corollary}

\theoremstyle{definition}
\newtheorem{definition}[equation]{Definition}

\theoremstyle{remark}
\newtheorem{remark}[equation]{Remark}

\newcommand{\add}{\operatorname{add}}

\newcommand{\Hom}{\operatorname{Hom}}

\newcommand{\Ho}{\mathrm{Ho}}

\newcommand{\ul}{\underline}
\newcommand{\xrto}{\xrightarrow}

\def\A{\mathcal A}

\def\C{\mathcal C}
\def\D{\mathcal D}
\def\E{\mathcal E}
\def\F{\mathcal F}
\def\G{\mathcal G}
\def\I{\mathcal I}

\def\P{\mathcal P}

\def\U{\mathcal U}
\def\T{\mathcal T}

\def\V{\mathcal V}
\def\X{\mathcal X}

\begin{document}

\title[\tiny{The realization of Verdier quotients as triangulated subfactor categories}]{The realization of Verdier quotients as triangulated subfactor categories}
\author [\tiny{Zhi-Wei Li}] {Zhi-Wei Li}

\date{\today}
\thanks{The author was supported by National Natural Science Foundation
of China (No.s 11671174 and 11571329).}

\email{zhiweili@jsnu.edu.cn}
\subjclass[2010]{18E35, 18E30}
\keywords{triangulated categories; torsion pairs; Verdier quotients}
\maketitle


\maketitle
\begin{center}
\tiny{School of Mathematics and Statistics, \ \ Jiangsu Normal University \\
Xuzhou 221116, Jiangsu, PR China.}
\end{center}
\begin{abstract} We give a method to realize Verdier quotients as triangulated subfactors of an arbitrary triangulated category. We show that Iyama-Yoshino triangulated subfactors are Verdier quotients under suitable conditions.
\end{abstract}

\setcounter{tocdepth}{1}

\section{Introduction}

Over the past decades, triangulated categories have been a unifying theory for many different parts of mathematics. One of the most important tools for triangulated categories is the Verdier quotient. However, one has little control over the morphisms of a Verdier quotient in general. One good way is to realize a Verdier quotient as a triangulated subfactor category like Buchweitz's equivalence between the singularity category and the stable category of Gorenstein projective modules of a Gorenstein ring \cite{Buchweitz87}. Recently, along this way, Iyama and Yang give a Buchweitz type theorem in triangulated categories by using silting reduction in \cite{Iyama-Yang} and Wei gives a more general analogous result in \cite{Wei2015}.

Let $\T$ be an arbitrary triangulated category and $\mathcal{N}$ a triangulated subcategory of $\T$. We use $\T/\mathcal{N}$ to denote the Verdier quotient of $\T$ with respect to $\mathcal{N}$. The aim of this paper is to give a general criterion to realize the Verdier quotient $\T/\mathcal{N}$ as a triangulated subfactor of $\T$ by introducing the notion of an {\it $\mathcal{N}$-localization triple} and the {\it Verdier condition}. This can be seen as an application of the homotopy theory of additive categories with suspensions developed in \cite{ZWLi2, ZWLi3}.
\vskip5pt
\noindent{\bf Theorem} (\ref{thm:localizationppt}) {\it Let $\T$ be a triangulated category and $\mathcal{N}$ a triangulated subcategory of $\T$. Assume that $(\U, \X, \V)$ is an $\mathcal{N}$-localization triple satisfying the Verdier condition. Then

$(\mathrm{i})$ $(\T, [1], \Delta, \V, \U, \X)$ is a pre-partial triangulated category in the sense of \cite[Definition 5.5]{ZWLi3}.
 
 $(\mathrm{ii})$ The subfactor $(\U\cap \V)/[\X]$ has a triangulated structure given by $(\T, [1], \Delta, \V, \U, \X)$.

$(\mathrm{iii})$ The embedding
$$(\C\cap \F)/[\X]\hookrightarrow \T/\mathcal{N}$$
is an equivalence of additive categories and preserves triangulated structures.}

\vskip5pt
\noindent We thank Dong Yang for informing us their recent analogous results in \cite{Iyama-Yang2} and Nakaoka's results in \cite{Nakaoka15}. In fact, we can deduce their results from the our main result directly; see Remark \ref{rem:IYandN}.

We now sketch the contents of the paper. In Section 2, we recall the definition of a special $\X$-monic closed subcategory and the constructions of triangulated subfactor categories given in \cite{ZWLi2}. In Section 3, we recall the notions of a localization triple and a pre-partial triangulated category discussed in \cite{ZWLi3}. In Section 4, we introduce the notion of an $\mathcal{N}$-localization triple for a triangulated subcategory $\mathcal{N}$ and prove our main result. The last Section is devoted to showing that Iyama-Yoshino triangulated subfactors are Verdier quotients under suitable conditions.

\vskip5pt

Throughout this paper, unless otherwise stated, that all subcategories of additive categories considered are full, closed under isomorphisms, all functors between additive categories are assumed to be additive.

\section{Triangulated subfactor categories}
In this section we fix a triangulated category $(\T, [1], \Delta)$ and an additive subcategory $\X$ of $\T$. We recall the notion of a special $\X$-monic closed subcategory and the criterion for the construction of triangulated subfactor categories in \cite{ZWLi2}.

\subsection*{Factor categories of additive categories} Let $\C$ be an additive category and $\X$ an additive subcategory of $\C$. We denote by $\C/[\X]$ the {\it factor} or {\it stable category} of $\C$ modulo $\X$. Recall that the objects of $\C/[\X]$ are the objects of $\C$, and for two objects $A$ and $B$ the Home space $\Hom_{\C/[\X]}(A,B)$ is the quotient $\Hom_\C(A,B)/\X(A,B)$, where $\X(A,B)$ is the subgroup of $\Hom_\C(A,B)$ consisting of those morphisms factorizing through an object in $\X$. Note that the stable category $\C/[\X]$ is an additive category and the canonical functor $\C\to \C/[\X]$ is an additive functor. For a morphism $f\colon A\to B$ in $\C$, we use $\ul{f}$ to denote its image in $\C/[\X]$.

A morphism $f\colon A\to B$ in $\C$ is said to be an {\it $\X$-monic} if the induced morphism $f^*=\Hom_\C(f, \X)\colon\\ \Hom_\C(B, \X)\to \Hom_\C(A, \X)$ is surjective. The notion of an {\it$\X$-epic} is defined dually. Recall that a morphism $f\colon A\to X$ in $\C$ is called an {\it $\X$-preenvelope} if $f$ is an $\X$-monic and $X\in \X$. Dually a morphism $g\colon X\to A$ is called an {\it $\X$-precover} if $g$ is an $\X$-epic and $X\in \X$.

\subsection*{Special $\X$-monic closed subcategories}  An additive subcategory $\C$ of $\T$ is said to be {\it special $\X$-monic closed} \cite[Definition 2.2]{ZWLi2} (note that the condition (b) of \cite[Definition 2.2]{ZWLi2} always holds in this setting) if

 (a)  $\X\subseteq \C$ and for each $A\in \C$, there is a triangle $A\stackrel{i}\to X\to U\to A[1]$ in $\Delta$ with $U\in \C$ and $i$ an $\X$-preenvelope.

 (b)  Assume that $A\stackrel{i}\to X\to U\to A[1]$ is in $\Delta$ with $U\in \C$ and $i$ an $\X$-preenvelope in $\C$. Then for any morphism $f\colon A\to B$ in $\C$, there is a triangle $A\xrightarrow{\left(\begin{smallmatrix}
	i \\
	f
	\end{smallmatrix}\right)}  X\oplus B \to N \to A[1]$ in $\Delta$ with $N\in \C$.

Dually, we can define the notion of a {\it special $\X$-epic closed} subcategory of $\T$.

\subsection*{The triangulation of subfactors}
 Let $\C$ be a special $\X$-monic closed subcategory. For each $A\in \C$, we {\it fix} a triangle $A\xrto{i^A} X^A\xrto{p^A} U^A\xrto{q^A} A[1]$ in $\Delta$ with $i^A$ an $\X$-preenvelope and $U^A\in \C$. Then there is a functor $\Sigma^\X\colon \C/[\X]\to \C/[\X]$ by sending each object $A$ to $U^A$ and each morphism $\ul{f}\colon A\to B$ to $\ul{\kappa}^f$, where the morphism $\kappa^f$ is defined by the following commutative diagram:
\begin{equation*}\label{kappaf}
\xy\xymatrixcolsep{2pc}\xymatrix@C10pt@R10pt{
A\ar[r]^{i^A}\ar[d]_f&X^A\ar[r]^{p^A}\ar[d]^{\sigma^f}&U^A\ar[r]^-{q^A}\ar[d]^{\kappa^f}& A[1]\ar[d]^{f[1]}\\
B\ar[r]^{i^B}&X^B\ar[r]^{p^B}&U^B\ar[r]^-{q^B}&B[1]}
\endxy
\end{equation*}
A {\it standard right triangle} in the subfactor category $\C/[\X]$ is an induced right $\Sigma^\X$-sequence $A\stackrel{\ul{f}}\to B\stackrel{\ul{g}}\to C\xrightarrow{\ul{\xi}(f,g)}\Sigma^\X(A)$ by a triangle $A\stackrel{f}\to B\stackrel{g}\to C\stackrel{h}\to A[1]$ in $\Delta$ with $f$ an $\X$-monic in $\C$ and $C\in \C$ which admits the following commutative diagram
\begin{equation*}\label{xif}
\xy\xymatrixcolsep{2pc}\xymatrix@C12pt@R14pt{
A\ar[r]^f\ar@{=}[d]&B\ar[r]^g\ar[d]^{\delta^f}&C\ar[r]^-h\ar[d]^{\xi(f,g)}& A[1]\ar@{=}[d]\\
A\ar[r]^{i^A}&X^A\ar[r]^{p^A}&U^A\ar[r]^-{q^A}&A[1]}
\endxy
\end{equation*}
 Denote by $\Delta^\X$ the class of right $\Sigma^\X$-sequences (called {\it distinguished right triangles}) in $\C/[\X]$ which are isomorphic to standard right triangles.

  Dually, if $\C$ is special $\X$-epic closed, for each object $A\in \C$, we {\it fix} a triangle $A[-1]\xrto{\nu_A} U_A\xrto{\iota_A} X_A\xrto{\pi_A} A$ in $\Delta$ with $\pi_A$ an $\X$-precover and $U_A\in \C$, then we can define an additive endofunctor $\Omega_\X\colon \C/[\X]\to \C/[\X]$, and the {\it standard left triangles} in $\C/[\X]$. Denote by $\nabla_\X$ the class of left $\Omega_\X$-sequences in $\C/[\X]$ which are isomorphic to standard left triangles.

By \cite[Corollary 4.11]{ZWLi2} and \cite[Theorem 3.2]{ZWLi2}, we have the following result.

 \begin{corollary} \label{thm:main}   $(\mathrm{i})$  If $\C$ is special $\X$-monic closed, then $(\C/[\X], \Sigma^\X, \Delta^\X)$ is a right triangulated category in the sense of \cite[Definition 1.1]{ABM}.
	
	$(\mathrm{ii})$ If $\C$ is special $\X$-epic closed, then $(\C/[\X], \Omega_\X, \nabla_\X)$ is a left triangulated category in the sense of \cite[Definition 2.2]{Beligiannis/Marmaridis94}.
\end{corollary}

\subsection*{Frobenius special $\X$-monic closed subcategories}
A special $\X$-monic closed subcategory $\C$ in $\T$ is said to be {\it Frobenius} \cite[Definition 6.1]{ZWLi2} (note that the the condition (iii) of \cite[Definition 6.1]{ZWLi2} holds automatically in this setting) if the following conditions hold.

(a)  For each $A\in \C$, there is a triangle $K\xrto{u} X\xrto{v} A\to K[1]$ in $\Delta$ with $u$ an $\X$-preenvelope in $\C$.

(b) For each $A\in \C$, $p^A$ is an $\X$-precover in the fixed triangle $A\xrto{i^A} X^A\xrto{p^A} U^A\xrto{q^A} A[1]$.
\vskip5pt

Dually, we can define a Frobenius special $\X$-epic closed subcategory of the triangulated category $\T$.

We have the following:

\begin{proposition} \ $($ \cite[Proposition 6.2 ]{ZWLi2}$)$ \label{thm:tricat}  $(\mathrm{i})$  If $\C$ is a Frobenius special $\X$-monic closed subcategory of the triangulated category $(\T, [1], \Delta)$. Then $(\C/[\X], \Sigma^\X, \Delta^\X)$ is a triangulated category.

$(\mathrm{ii})$   If $\C$ is a Frobenius special $\X$-epic closed subcategory of the triangulated category $(\T, [1], \Delta)$. Then $(\C/[\X], \Omega_\X, \nabla_\X)$ is a triangulated category.
\end{proposition}

\section{Pre-partial triangulated categories}

In this section we recall the notions of a localization triple and a pre-partial triangulated category in a triangulated category introduced in \cite{ZWLi3}. We fix a triangulated category $(\T, [1], \Delta)$ with an additive subcategory $\X$.
\subsection*{Localization triples}
 Assume $\U, \V$ are two additive subcategories of $\T$ such that $\X\subseteq \U\cap \V$. Let $\A$ be an additive subcategory of $\T$. We denote by $$\U^{\perp_{\A/[\X]}}=\{W\in \A \ | \ \Hom_{\A/[\X]}(\U, W)=0\} \ \mbox{and} \ {^{\perp_{\A/[\X]}}\V}=\{W\in \A \ | \ \Hom_{\A/[\X]}(W, \V)=0\}.$$

\begin{definition} $($\cite[Definition 3.2]{ZWLi3}$)$ \label{htriple}  The triple $(\U, \X, \V)$ is called a {\it localization triple} of $\A$ if $\U, \V\subseteq \A$ and the following conditions hold:

(a) For each $A\in \A$, there is a triangle $A[-1]\to W_A\xrto{\omega_A} Q(A)\xrto{r_A}A$ in $\Delta$ such that $r_A$ is a $\U$-precover and $W_A\in \U^{\perp_{\A/[\X]}}$.

(b) For each $A\in \A$, there is a triangle $A\xrto{j^A} R(A)\xrto{\tau^A} W^A\to A[1]$ in $\Delta$ such that $j^A$ is a $\V$-preenvelope and $W^A\in {^{\perp_{\A/[\X]}}\V}$.

(c) If $A\in \V$, then $Q(A)\in \U\cap \V, W_A\in (\U\cap \V)^{\perp_{\V/[\X]}}$, and if $A\in \U$, then $R(A)\in \U\cap \V, W^A\in {^{\perp_{\U/[\X]}}(\U\cap\V)}$.
\end{definition}

\begin{remark} \label{remark:adjoints}
If $(\U, \X, \V)$ is a localization triple of $\A$, recall that, by \cite[Remark 3.5]{ZWLi3}, then the embedding $\U/[\X]\hookrightarrow \A/[\X]$ has a right adjoint $Q\colon \A/[\X]\to \U/[\X]$ which sends an object $A$ to $Q(A)$ and a morphism $\ul{f}\colon A\to B$ to $\ul{\check{f}}$ which satisfies $r_B\circ \check{f}=f\circ r_A$. Similarly, the embedding $\V/[\X]\to \A/[\X]$ has a left adjoint $R\colon \A/[\X]\to \V/[\X]$ which sends an object $A$ to $R(A)$ and a morphism $\ul{f}\colon A\to B$ to $\ul{\hat{f}}$ which satisfies $\hat{f}\circ j^A=j^B\circ f$.
\end{remark}

 Let $\mathrm{Mor}(\A)$ be the class of morphisms in $\A$. Then we define
\begin{equation*} \label{we}
\mathcal{S}_{[\X]}=\{s\in \mathrm{Mor}(\A) \ | \ RQ(\ul{s})=\ul{\hat{\check{s}}} \ \mbox{is an isomorphism in } \ (\U\cap \V)/[\X]\}.
 \end{equation*}
 By \cite[Theorem 3.8, Remark 3.9]{ZWLi3}$)$, the Gabriel-Zisman localization $\A[\mathcal{S}_{[\X]}^{-1}]$ of $\A$ with respect to $\mathcal{S}_\X$ exists and is equivalent to the subfactor category $(\U\cap\V)/[\X]$. We call $\A[\mathcal{S}_{[\X]}^{-1}]$ the {\it homotopy category} of the localization triple $(\U, \X, \V)$ and denoted by $\Ho(\U, \X, \V)$.

 Assume that $\C$ is a special $\X$-monic closed subcategory of $\T$. Then the subfactor category $\C/\X$ has a right triangulated structure $(\Sigma^\X, \Delta^\X)$ by Theorem \ref{thm:main}. For any morphism $f\colon A\to B$ in $\C$, there is a right triangle $A\xrto{\ul{f}} B\xrto{\ul{g}} C\xrto{\ul{h}}\Sigma^\X(A)$ in $\Delta^\X$.

\begin{definition} \cite[Definition 4.1]{ZWLi3} \label{stabilizing} An additive subcategory $\G$ of $\C$ is said to be {\it stabilizing} in $\C$ if $(\C, \X, \G)$ is a localization triple in $\C$ and for any diagram of the form
 \[\xy\xymatrixcolsep{2pc}\xymatrix@C16pt@R16pt{
A\ar[r]^-{\ul{f}}\ar[d]_{\ul{j}^A}&B\ar[r]\ar[d]^{\ul{j}^B}&C\ar[r]\ar@{.>}[d]^{\ul{t}}& \Sigma^\X(A)\ar[d]^{\Sigma^\X(\ul{j}^A)}\\
R(A)\ar[r]^-{R(\ul{f})}&R(B)\ar[r]&D\ar[r]&\Sigma^\X(R(A))}
\endxy\]
where the rows are right triangles in $\Delta^\X$, there is a morphism $\ul{t}\colon C\to D$ such that the whole diagram commutative and $R(\ul{t})$ is an isomorphism in $\G/\X$. Where $R\colon \C/\X\to \G/\X$ is a left adjoint of the embedding $E^\C\colon \G/\X\hookrightarrow \C/\X$ as constructed in Remark \ref{remark:adjoints}.
\end{definition}

Dually, if $\F$ is a special $\X$-epic closed subcategory of $\T$, then the subfactor category $\F/\X$ has a left triangulated structure $(\Omega_\X, \nabla_\X)$. We say that an additive subcategory $\G$ of $\F$ is {\it stabilizing}, if $(\G, \X, \F)$ is a localization triple in $\F$ and for any diagram of the form
\[\xy\xymatrixcolsep{2pc}\xymatrix@C16pt@R16pt{
\Omega_\X(Q(B))\ar[r] \ar[d]_{\Omega(\ul{r}_B)} & K\ar[r]\ar@{.>}[d]^{\ul{s}}&Q(A)\ar[r]^-{Q(\ul{f})}\ar[d]^-{\ul{r}_A}& Q(B)\ar[d]^{\ul{r}_B}\\
\Omega_\X(B)\ar[r]&L\ar[r]&A\ar[r]^-{\ul{f}}&B}
\endxy\]
where the rows are the distinguished left triangles in $\nabla_\X$, there is a morphism $\ul{s}\colon K\to L$ such that the whole diagram commutative and $Q(\ul{s})$ is an isomorphism in $\G/\X$. Where $Q\colon \F/\X\to \G/\X$ is a right adjoint of the embedding $E_\F\colon \G/\X\hookrightarrow \F/\X$ as constructed by Remark \ref{remark:adjoints}.

\begin{definition} (\cite[Lemma 5.9, Definition 5.5]{ZWLi3}) Let $(\T, [1], \Delta)$ be a triangulated category. Assume that $\X, \C, \F$ are three additive subcategories of $\T$. Then the six-tuple $(\T, [1], \Delta, \F, \C, \X)$ is called a {\it pre-partial triangulated category} if $\C$ is special $\X$-monic closed, $\F$ is special $\X$-epic closed and $\C\cap \F$ is stabilizing in both $\C$ and $\F$.
\end{definition}

If $(\T, [1], \Delta, \F, \C, \X)$ is a pre-partial triangulated category, then we have two endofunctors $G^\X=R\circ \Sigma^\X, E^\C$ and $H_\X=Q\circ \Omega_\X\circ E_\F$ of $(\C\cap \F)/[\X]$. We call the right $G^\X$-sequence $A\xrto{\ul{f}} B\xrto{R(\ul{g})} R(C)\xrto{R(\ul{h})} G^\X(A)$ in $(\C\cap \F)/[\X]$ a {\it standard right triangle} if  $A\xrto{\ul{f}} B\xrto{\ul{g}} C\xrto{\ul{h}}\Sigma^\X(A)$ is a distinguished right triangle in $\Delta^\X$. We use $\Delta^\X(\C\cap \F)$ to denote the class of right $G^\X$-sequences which are isomorphic to standard right triangles in $(\C\cap \F)/[\X]$. Dually, we call the left $H_\X$-sequence $H_\X(B)\xrto{Q(\ul{v})} Q(K)\xrto{Q(\ul{u})} A\xrto{Q(\ul{f})} B$ a {\it standard left triangle} in $(\C\cap \F)/[\X]$ if there is a left triangle $\Omega_\X(B)\xrto{\ul{v}} K\xrto{\ul{u}} A\xrto{\ul{f}} B$ in $\nabla_\X$. We use $\nabla_\X(\C\cap \F)$ to denote the class of the left $H_\X$-sequences which are isomorphic to standard left triangles in $(\C\cap \F)/[\X]$.

\begin{theorem}$($ \cite[Theorem 5.7 ]{ZWLi3}$)$\label{thm:pretricat} If $(\T, [1], \Delta, \F, \C, \X)$ is a pre-partial triangulated category, then  $( (\C\cap \F)/[\X],
G^\X, H_\X,  \Delta^\X(\C\cap \F), \nabla_\X(\C\cap \F))$ is a pre-triangulated category in the sense of \cite[Definition II.1.1]{Beligiannis/Reiten07}.
\end{theorem}

\section{Verdier quotients vs triangulated subfactors}
In this section, we fix a triangulated category $(\T, [1], \Delta)$ and a triangulated subcategory $\mathcal{N}$ of $\T$. We introduce the notion of an $\mathcal{N}$-localization triple and prove our main result.

\subsection*{$\mathcal{N}$-localization triples}

\begin{definition} \label{defn:nlocaltrip}  A localization triple $(\U, \X, \V)$ in $\T$ is called an {\it $\mathcal{N}$-localization triple} if $\X=\U\cap \V\cap \mathcal{N}$ and $\U^{\perp_{\T/[\X]}}, {^{\perp_{\T/[\X]}}\V}\subseteq \mathcal{N}$.

\end{definition}

An $\mathcal{N}$-localization triple $(\U, \X, \V)$ is said to satisfy the {\it Verdier condition} if for any triangle $A\xrto{s} B\to N\to A[1]$ in $\Delta$ with $A, B\in \U\cap \V, N\in \mathcal{N}$, then $\ul{s}$ is an isomorphism in $(\U\cap \V)/[\X]$.

\begin{lemma} \label{lem:verdiercondition} $(\mathrm{i})$ If $(\U, \X, \V)$ is an $\mathcal{N}$-localization triple such that $\U\cap \V$ is Frobenius special $\X$-monic closed, then it satisfies the Verdier condition.

$(\mathrm{ii})$  If $(\U, \X, \V)$ is an $\mathcal{N}$-localization triple such that $\U\cap \V$ is Frobenius special $\X$-epic closed, then it satisfies the Verdier condition.
\end{lemma}

\begin{proof} (i) Assume that we have a triangle $A\stackrel{s}\to B\to N\to A[1]$ in $\Delta$ such that $A, B\in \U\cap \V$ and $N\in \mathcal{N}$. Since $\U\cap \V$ is special $\X$-monic closed, there is a triangle $A\xrto{i^A} X^A\to U^A\to A[1]$ in $\Delta$ with $X^A\in \X$ and $U^A\in \U\cap \V$. By the cobase change of triangles in $\Delta$, there is a commutative diagram of triangles:
\[\xy\xymatrixcolsep{2pc}\xymatrix@C12pt@R12pt{
A\ar[r]^s\ar[d]_{i^A}&B\ar[d]\ar[r]&N\ar[r]\ar@{=}[d]&A[1]\ar[d]\\
X^A\ar[r]\ar[d]&B'\ar[r]\ar[d]&N\ar[r]&X^A[1]\\
U^A\ar@{=}[r]\ar[d]&U^A\ar[d]& &\\
A[1]\ar[r]&B[1]& &}
\endxy\]
Since $\U\cap \V$ is special $\X$-monic closed and $\mathcal{N}$ is closed under extensions we have $B'\in \U\cap \V\cap \mathcal{N}=\X$. Thus the triangle $A\xrightarrow{\left(\begin{smallmatrix}
i^A\\
s
\end{smallmatrix}\right)} X^A\oplus B\to B'\to A[1]$ induces a standard triangle $A\stackrel{\ul{s}}\to B\to 0\to \Sigma^\X(A)$ in $(\U\cap \V)/[\X]$. Thus $\ul{s}$ is an isomorphism.

The statement (ii) can be proved dually. \end{proof}

Recall that a pair $(\C, \D)$ of additive subcategories of $\T$ is called a {\it torsion pair} (also called a {\it torsion theory} in \cite[Definition 2.2]{Iyama-Yoshino}) if $\Hom_\T(\C, \D)=0$ and for each $T\in \T$, there is a triangle $C\to T\to D\to C[1]$ in $\Delta$ with $C\in \C$ and $D\in \D$.

\begin{lemma} \label{lem:stabilizing}  Let $(\U, \U^{\perp})$ and $(^{\perp}\V, \V)$ be torsion pairs in $\T$. Assume that $\X=\U\cap \V\cap \mathcal{N}$ is closed under direct summands, $\U^{\perp}[-1]\subseteq \V\cap \mathcal{N}$ and $ {^{\perp}\V}[1]\subseteq \U\cap \mathcal{N}$. Then

$(\mathrm{i})$ $(\U,\X,\V)$ is an $\mathcal{N}$-localization triple.

$(\mathrm{ii})$ If ${^{\perp}\V}[1]=\U\cap \mathcal{N} $, then $(\U,\X,\V)$ satisfies the Verdier condition.

$(\mathrm{iii})$ If $\U^{\perp}[-1]=\V\cap \mathcal{N} $, then $(\U,\X,\V)$ satisfies the Verdier condition.
\end{lemma}
\begin{proof} (i) By the assumption, we only need to prove that Definition \ref{htriple} (b) holds. Since  $(\U, \U^{\perp})$ and $(^{\perp}\V, \V)$ are torsion pairs in $\T$, for all $A\in \T$, there are triangles $A[-1]\to W_A[-1]\to Q(A)\xrto{r_A} A $ with $Q(A)\in \U$ and $W_A\in \U^{\perp}$, and $A \xrto{j^A} R(A)\to W^A[1]\to A[1]$ with  $R(A)\in \V$ and $W^A\in {^{\perp}\V}$. Since $\U^{\perp}[-1]\subseteq \V,  {^{\perp}\V}[1]\subseteq \U$, and $\U, \V$ are closed under extensions, we know that Definition \ref{htriple} (c) holds. Now assume there is a morphism $f\colon C\to G$ with $C\in \U, G\in \U^{\perp}[-1]$, then there is a triangle $G[-1]\to W_G[-1]\to Q(G)\xrto{r_G} G$ with $Q(G)\in \U$ and $W_G\in \U^{\perp}$. Since $\U^{\perp}$ is closed under extensions, we have $Q(G)\in \U\cap \U^{\perp}[-1]\subseteq \X$. Since $r_{G}$ is a $\C$-epic, the morphism $f$ factors through $r_G$, then $\U^{\perp}[-1]\subseteq \U^{\perp_{\T/[\X]}}$. Similarly we can prove that ${^{\perp}\V}[1]\subseteq {^{\perp_{\T/[\X]}}\V}$. From these, we can show that $(\U, \X, \V)$ is a localization triple in $\T$. By the above proof and the assumption we know that $(\U, \X, \V)$ is an $\mathcal{N}$-localization triple.

(ii) Let $\I=\U\cap \U^{\perp}[-1]\subseteq\X$. Then $\U$ is special $\I$-monic closed by \cite[Corollary 4.14 (i)]{ZWLi2}. Assume that we have a triangle $A\stackrel{s}\to B\to N\to A[1]$ in $\Delta$ such that $A, B\in \U\cap \V$ and $N\in \mathcal{N}$. Similarly to the proof of Lemma \ref{lem:verdiercondition} (i), we can get a triangle $A\xrightarrow{\left(\begin{smallmatrix}
i\\
s
\end{smallmatrix}\right)} I\oplus B\to B'\to A[1]$ with $I\in \I$ and $B'\in \U\cap \mathcal{N}={^{\perp}\V}[1]$. Thus $I\oplus B\cong A\oplus B'$ and then $B'\in \U\cap \V\cap \mathcal{N}=\X$. Therefore, $\ul{s}$ is an isomorphism in $(\U\cap \V)/[\X]$. So $(\U, \X, \V)$ satisfies the Verdier condition.

Dually, we can prove (iii).
\end{proof}

\subsection*{The Verdier quotients}
Recall that the Verdier quotient $\T/\mathcal{N}$ is the localization of $\T$ with respect to the class $\mathcal{S}_{\mathcal{N}}$ of morphisms of $\T$ defined by
$$\mathcal{S}_{\mathcal{N}}=\{s\in \mathrm{Mor}(\T) \ |\ \exists \ A\xrto{s}B\to N\to A[1]\in \Delta \ \mbox{with}\  N\in \mathcal{N}\}.$$

\begin{proposition} \label{prop:verdierhc}  Let $(\T, [1], \Delta)$ be a triangulated category and $\mathcal{N}$ a triangulated subcategory of $\T$. Assume that $(\U,\X,\V)$ is an $\mathcal{N}$-localization triple of $\T$ satisfying the Verdier condition. Then $\Ho(\U, \X, \V)=\T/{\mathcal{N}}$.

\end{proposition}
\begin{proof}  By \cite[Theorem 3.8, Remark 3.9]{ZWLi3}$)$, we only need to prove that $\mathcal{S}_{[\X]}=\mathcal{S}_{\mathcal{N}}$. In fact, since $(\U, \X, \V)$ is an $\mathcal{N}$-localization triple in $\T$, for each $A\in \T$, we have two triangles in $\Delta$: $A[-1]\to W_A\xrto{\omega_A} Q(A)\xrto{r_A} A$ and $A\xrto{j^A} R(A)\xrto{\tau^A} W^A\to A[1]$ such that $r_A$ is a $\U$-precover with $W_A\in \U^{\perp_{\T/[\X]}}\subseteq \mathcal{N}$ and $j^A$ a $\V$-preenvelope with $W^A\in {^{\perp_{\T/[\X]}}\V}\subseteq \mathcal{N}$. Thus the morphisms $r_A$ and $j^A$ are in $\mathcal{S}_{\mathcal{N}}$. So for any morphism $s\colon A\to B$ in $\T$, $s\in \mathcal{S}_{\mathcal{N}}$ if and only if $\hat{\check{s}}\colon R(Q(A))\to R(Q(B))$ is in $\mathcal{S}_{\mathcal{N}}$ (which satisfies two out of three property). Similarly, $s\in \mathcal{S}_{[\X]}$ if and only if $\hat{\check{s}}$ is in $\mathcal{S}_{[\X]}$. Thus in order to prove that $\mathcal{S}_\X=\mathcal{S}_{[\mathcal{N}]}$, we may only consider the morphisms in $\U\cap \V$. Assume that we have a morphism $s\colon A\to B$ in $\U\cap \V$, extend it to a triangle $A\stackrel{s}\to B\to C\to A[1]$ in $\Delta$. If $s\in \mathcal{S}_{\mathcal{N}}$, then $C\in \mathcal{N}$, thus by the Verdier condition, we know that $s\in \mathcal{S}_{[\X]}$. Conversely, if $s\in \mathcal{S}_{[\X]}$, i.e. $\ul{s}$ is an isomorphism in $(\U\cap \V)/[\X]$, then there is a morphism $t\colon B\to A$ such that $1_A-ts$ factors through some $X\in \X$, say $1_A-ts=u\circ v$ where $v\colon A\to X$ and $u\colon X\to A$.
So we have a triangle $A\xrto{\left(\begin{smallmatrix}
v \\
s
\end{smallmatrix}\right)}X\oplus B\to X'\xrto{0} A[1]$ in $\Delta$ with $X'\in \X$. By the octahedron axiom, there is a triangle $X'\to C\to X[1]\to X'[1]$ in $\Delta$ and thus $C\in \mathcal{N}$ since $\X\subseteq \mathcal{N}$.
\end{proof}

\begin{corollary} \label{cor:verdiersubfactor} Let $(\T, [1], \Delta)$ be a triangulated category and $\mathcal{N}$ a triangulated subcategory of $\T$. If $(\U, \X, \V)$ is an $\mathcal{N}$-localization triple satisfying the Verdier condition, then we have an equivalence of additive categories $(\U\cap \V)/[\X]\stackrel{\sim}\hookrightarrow \T/\mathcal{N}$.
\end{corollary}

\begin{proof} This follows from Theorem \ref{prop:verdierhc} and \cite[Theorem 3.8, Remark 3.9]{ZWLi3}$)$ directly.\end{proof}

\begin{theorem} \label{thm:localizationppt} Let $(\U, \X, \V)$ be a $\mathcal{N}$-localization triple satisfying the Verdier condition. Assume that $\U$ is special $\X$-monic closed and $\V$ is special $\X$-epic closed. Then

$(\mathrm{i})$  $(\T, [1], \Delta, \V, \U, \X)$ is a pre-partial triangulated category such that the induced pre-triangulated structure of $(\U\cap \V)/[\X]$ is a triangulated structure.

$(\mathrm{ii})$ The equivalence $(\U\cap \V)/[\X]\stackrel{\sim}\hookrightarrow \T/\mathcal{N}$ of additive categories preserves triangulated structures.
\end{theorem}
\begin{proof} (i) Since $(\U, \X, \V)$ is a localization triple of $\T$, then for each $A\in \U$, there is a triangle $A\xrto{j^A} R(A)\xrto{\tau^A} W^A\to A[1]$ in $\Delta$ such that $j^A$ is a $\V$-preenvelope, $R(A)\in \U\cap\V$ and $W^A\in {^{\perp_{\U/[\X]}}(\U\cap \V)}$. Thus $(\U, \X, \U\cap \V)$ is a localization triple of $\U$. Since ${^{\perp_{\U/[\X]}}(\U\cap \V)}$ is extension closed, we know that ${^{\perp_{\U/[\X]}}(\U\cap \V)}$ is special $\X$-monic closed by \cite[Lemma 4.12]{ZWLi2} and noting that $(\U\cap \V)\cap {^{\perp_{\U/[\X]}}(\U\cap \V)}=\X$. Therefore, ${^{\perp_{\U/[\X]}}(\U\cap \V)}/[\X]$ is a right triangulated subcategory of $\U/[\X]$. So, if $W_1\xrto{r} W_2\to W_3\to W_1[1]$ is a triangle in $\Delta$ with $r$ an $\X$-monic in $\U$ and $W_1\in {^{\perp_{\U/[\X]}}(\U\cap \V)}$, then $W_3\in {^{\perp_{\U/[\X]}}(\U\cap \V)}$ iff $W_2\in {^{\perp_{\U/[\X]}}(\U\cap \V)}$. Moreover, if $A\xrto{f} B\to W\to A[1]$ is a triangle in $\Delta$ with $f$ an $\X$-monic in ${^{\perp_{\U/[\X]}}(\U\cap \V)}$ and $W\in {^{\perp_{\U/[\X]}}(\U\cap \V)}$, then $\ul{f}$ is an isomorphism in $(\U\cap \V)/[\X]$ since $(\U, \X, \V)$ satisfies the Verdier condition and ${^{\perp_{\U/[\X]}}(\U\cap \V)}=\U\cap {^{\perp_{\T/[\X]}}\V}\subseteq \U\cap \mathcal{N}$. Thus $\U\cap \V$ is stabilizing in $\U$ by \cite[Lemma 4.2]{ZWLi3}. Dually, we can prove that $\U\cap \V$ is stabilizing in $\V$. So $(\T, [1], \Delta, \V, \U, \X)$ is a pre-partial triangulated category.

By Theorem \ref{thm:pretricat}, $(\U\cap \V)/[\X]$ admits a pre-triangulated structure $(G^\X, H_\X, \Delta^\X(\U\cap \V), \nabla_\X(\U\cap \V))$. By the proof of \cite[Theorem 5.7]{ZWLi3}, the unit of the adjoint pair $(G^\X, H_\X)$ is given by $-Q(\ul{u})$, where $u$ fits the following diagram of triangles in $\Delta$
\[\xy\xymatrixcolsep{2pc}\xymatrix@C18pt@R16pt{
\Sigma^\X(A)[-1]\ar[r]^-{-q^A[-1]}\ar[d]_{j^{\Sigma^\X(A)}[-1]}&A\ar[r]^-{i^A} \ar[d]^{u}&X^A\ar[r]^-{p^A}\ar[d]^{\delta}& \Sigma^\X(A)\ar[d]^{j^{\Sigma^\X(A)}}\\
G^\X(A)[-1]\ar[r]^-{\nu_{G^\X(A)}}&\Omega_\X(G^\X(A))\ar[r]^-{\iota_{G^\X(A)}}&X_{G^\X(A)}\ar[r]^-{\pi_{G^\X(A)}}& G^\X(A)}
\endxy\]
By \cite[Proposition 1.1.11]{BBD}, the above diagram can be extended to the following diagram of triangles
\[\xy\xymatrixcolsep{2pc}\xymatrix@C10pt@R10pt{
A\ar[r]\ar[d]_{u}&X^A\ar[r]\ar[d]^{\delta}&\Sigma^\X(A)\ar[r]\ar[d]^{j^{\Sigma^\X(A)}}& A[1]\ar[d]^{\mu[1]}\\
\Omega_\X(G^\X(A))\ar[r]\ar[d]&X_{G^\X(A)}\ar[r]\ar[d]& G^\X(A)\ar[r]\ar[d]&\Omega_\X(G^\X(A))[1]\ar[d]\\
L\ar[r]\ar[d] &M\ar[r]\ar[d]&W^{\Sigma^\X(A)}\ar[r]\ar[d]&L[1]\ar[d]\\
A[1]\ar[r]&X^A[1]\ar[r]&\Sigma^\X(A)[1] \ar[r]& A[2]}
\endxy\]
Since $\X\subseteq \mathcal{N}$, $W^{\Sigma^\X(A)}\in {^{\perp_{\U/[\X]}}(\U\cap \V)}=\U\cap{^{\perp_{\T/[\X]}}\V} \subseteq \mathcal{N}$, we know that $L\in \mathcal{N}$. Therefore $u \in \mathcal{S}_{\mathcal{N}}$ which is $\mathcal{S}_{[\X]}$ by the proof of Lemma \ref{prop:verdierhc}. So $Q(\ul{u})$ is an isomorphism.  Similarly, we can prove that the counit of the adjoint pair $(G^\X, H_\X)$ is also an isomorphism. So $G^\X$ is an auto-equivalence of $(\U\cap \V)/[\X]$ and thus $(G^\X, \Delta^\X(\U\cap \V))$ is a triangulated structure.

(ii) By Lemma \ref{lem:verdiercondition} and Corollary \ref{cor:verdiersubfactor}, we know that the embedding $E\colon (\U\cap \V)/[\X]\hookrightarrow \T/\mathcal{N}$ is an equivalence of additive categories. So we only need to prove that $E$ preserves triangles. Applying the triangle functor $\gamma'\colon \T\to \T/\mathcal{N}$ (the localization of $\T$ with respect to $\mathcal{S}_{\mathcal{N}}$) to the triangles $A\to X^A\to \Sigma^\X(A)\xrto{q^A} A[1]$ with $A\in \U$ and $\Sigma^\X(A)\xrto{j^{\Sigma^\X(A)}}R(\Sigma^\X(A))\to W^{\Sigma^\X(A)}\to \Sigma^\X(A)[1]$ with $R(\Sigma^\X(A))\in \U\cap \V, W^{\Sigma^\X(A)}\in {^{\perp_{\T/[\X]}}\V}$,  we get an isomorphism $\gamma'(q^A)\circ \gamma'(j^{\Sigma^\X(A)})^{-1}\colon R(\Sigma^\X(A))\to A[1]$ in $\T/\mathcal{N}$, which defines a natural isomorphism $E\circ G^\X=E\circ R\circ \Sigma^\X\cong [1]\circ E $. Let $A\xrto{\ul{f}} B\xrto{R(\ul{g})} R(C)\xrto{R(\ul{h})} G^\X(A)$ be a triangle in $(\U\cap \V)/[\X]$ which is induced by the left triangle $A\xrto{\ul{f}} B\xrto{\ul{g}} C\xrto{\ul{h}} \Sigma^\X(A)$ in $\U/[\X]$ with $A, B\in \U\cap \V$. Without loss of generality, we may assume that it is induced by the following commutative diagram of triangles in $\Delta$:
\[\xy\xymatrixcolsep{2pc}\xymatrix@C12pt@R12pt{
A\ar[r]^f\ar@{=}[d]&B\ar[d]\ar[r]^g&C\ar[r]\ar[d]^h&A[1]\ar@{=}[d]\\
A\ar[r]&X^A\ar[r]&U^A\ar[r]^{q^A}&A[1]}
\endxy\]
which shows that $A\xrto{\gamma'(f)} B\xrto{\gamma'(g)}C\xrto{\gamma'(q^A\circ h)} A[1]$ is a triangle in $\T/\mathcal{N}$. Thus $A\xrto{\gamma'(f)} B\xrto{\gamma'(R(g))}R(C)\xrto{\gamma'(q^A)\circ \gamma'(j^{\Sigma^\X(A)})^{-1}\circ \gamma'(R(h))} A[1]$ is a triangle in $\T/\mathcal{N}$ (we remind the reader that we have $j^C\circ g=R(g)$ and $R(h)\circ j^C=j^{\Sigma^\X(A)}\circ h$). So $E$ is a triangle functor.

\end{proof}

By the above Theorem and Lemma \ref{lem:verdiercondition}, we have

\begin{corollary} \label{thm:FrobeniusVerdier} Let $(\T, [1], \Delta)$ be a triangulated category and $\mathcal{N}$ a triangulated subcategory of $\T$. If $(\U, \X, \V)$ is an $\mathcal{N}$-localization triple such that $\U\cap \V$ is Frobenius special $\X$-monic or $\X$-epic closed, then there is a triangle equivalence $$(\U\cap \V)/[\X]\stackrel{\sim}\hookrightarrow \T/\mathcal{N}.$$
\end{corollary}


\begin{remark} \label{rem:IYandN}(i) Let $\T$ be a triangulated category and $\mathcal{N}$ a thick triangulated subcategory of $\T$ satisfying the following conditions:

$(\mathrm{a})$  there are torsion pairs $(\C, \C^{\perp})$ and $({^{\perp}\D}, \D)$ in $\T$;

$(\mathrm{b})$ $(\C, \D)$ is a torsion pair in $\mathcal{N}$. 

\noindent Let $\X=\C[1]\cap \D$ and $\mathcal{Z}=\C^\perp\cap {^\perp\D}[1]$. Then $({^\perp\D}[1], \X, \C^\perp)$ is an $\mathcal{N}$-localization triple satisfies the Verdier condition by taking $\U={^\perp\D}[1]$ and $\V=\C^\perp$ in Lemma \ref{lem:stabilizing}, so we can get \cite[Theorem 1.1, Theorem 1.2 (a)]{Iyama-Yang2} by Theorem \ref{thm:localizationppt}.

(2) Assume that $\T$ is a triangulated category and there are two torsion pairs $(\C, \D[1])$ and $(\E[-1], \F)$ satisfying the following conditions:

$(\mathrm{a})$  $\E\subseteq \C$, $\D\subseteq \F$ and $\C\cap\D =\E\cap \F$;

$(\mathrm{b})$ there is a thick triangulated subcategory $\mathcal{N}$ such that $ \C\cap \F\cap\mathcal{N}=\C\cap\D, \C\cap \mathcal{N}=\E$ and $\F\cap \mathcal{N}=\D$.

\noindent Then $(\C, \C\cap \D, \F)$ is an $\mathcal{N}$-localization triple satisfies the Verdier condition by taking $\U=\C, \V=\F$ and $\X=\C\cap \D$ in Lemma \ref{lem:stabilizing}, so we can get \cite[Proposition 6.10, Corollary 6.13]{Nakaoka15} by Theorem \ref{thm:localizationppt}.
\end{remark}

\section{Iyama-Yoshino triangulated subfactors} Let $(\T,[1], \Delta)$ be a triangulated category. Let $\X\subseteq C$ be two additive subcategories of $\T$ closed under direct summands. We use $\langle\X\rangle$ to denote the smallest triangulated subcategory of $\T$ containing $\X$.

Recall that $(\C, \C)$ is said to be an {\it $\X$-mutation pair} \cite[Definition 2.5]{Iyama-Yoshino} if $\C$ is extension-closed, $\Hom_\T(\X[-1], \C)=0=\Hom_\T(\C, \X[1])$, and for any object $A\in \C$, there exist triangles $A[-1]\to K_A\to X_A\to A$ and $ A\to X^A\to K^A\to A[1]$ in $\bigtriangleup$ such that $
X_A, X^A\in\X$ and $K_A, K^A\in \C$. In this case, by \cite[Example 6.4 (ii), Proposition 6.5]{ZWLi2}, we know that $\C$ is Frobenius special $\X$-monic closed. Thus we have the following corollary by Corollary \ref{thm:FrobeniusVerdier}.
\begin{corollary} \label{cor:Verdier vs IY} Assume that $(\C, \C)$ forms an $\X$-mutation pair. If there is an $\langle\X\rangle$-localization triple $(\U, \X, \V)$ such that $\U\cap \V=\C$, then there is a triangle equivalence
$$\C/[\X]\stackrel{\sim}\hookrightarrow \T/\langle\X\rangle.$$
\end{corollary}

\subsection*{A theorem of Iyama-Yang} Let $(\T, [1], \Delta)$ be a triangulated category. Let $\P$ be a {\it presilting} (also called {\it semi-selforthogonal}) subcategory (i.e., $\Hom_\T(\P, \P[i])=0$ for all $i>0$) of $\T$. Assume that $\P$ is closed under direct summands and satisfies the following two conditions.

(P1) $\P$ is covariantly finite in ${^{\perp}\P[>0]}$ and contravariantly finite in ${\P[<0]^{\perp}}$.

(P2) For any $A\in \T$, $\Hom_\T(A, \P[i])=0=\Hom_\T(\P, A[i])$ for $i\gg 0$.

\noindent  Let $\mathcal{Z}=(^{\perp}\P[>0]\cap \P[<0]^{\perp})$. Then $(\mathcal{Z}, \mathcal{Z})$ forms a $\P$-mutation pair and $({^{\perp}\P[>0]}, \P, \P[<0]^{\perp})$ is a $\langle \P\rangle$-localization triple by the proof of \cite[Proposition 3.2]{Iyama-Yang}. Thus by Corollary \ref{cor:Verdier vs IY}, we have

\begin{corollary} \label{Iyama-Yang's result} $($\cite[Theorem 3.6]{Iyama-Yang}$)$ There is a triangle equivalence $\mathcal{Z}/[\P]\stackrel{\sim}\hookrightarrow\T/\langle\P\rangle$.
\end{corollary}

The above corollary contains Buchweitz' equivalence \cite[Theorem 4.4.1(2)]{Buchweitz87}. In fact, assume that $\T=\mathrm{D^b}(R)$ is the bounded derived category of a Gorenstein ring $R$, and $\P=\mathrm{add}R$ is the subcategory of finitely generated projective right $R$-modules. Then $\mathcal{Z}$ is the subcategory of Gorenstein projective right $R$-modules.

\subsection*{A theorem of Wei}
Let $(\T, [1], \Delta)$ be a triangulated category. Let $\omega$ be a semi-selforthogonal subcategory of $\T$. Given a subcategory $\C$ of $\T$, we define the subcategory $\widehat{\C}$ as the class of all objects $T$ such that there are triangles $T_{i+1}\to C_i\to T_i\to T_{i+1}[1]$ in $\Delta$ for some $r\geq 0$ and all $0 \leq i\leq r$ satisfying $T_0=T, T_{r+1}=0$ and $C_i\in \C$ for each $i$. The subcategory $\widecheck{\C}$ is defined dually. Let $\omega$ be a subcategory of $\T$. We use ${\rm add} \omega$ to denote the class of all direct summands of finite direct sums of copies of objects in $\omega$. Let $\X_\omega$ be the subcategory of $\T$ consisting of all objects $T$ satisfying that there are triangles $T_i\to M_i\to T_{i+1}\to T_i[1]$ in $\Delta$ for all $i\geq 0$, such that $T_0=T$ and for all $i>0$, $M_i\in {\rm add}\omega$ and $\Hom_\T(T_i, \omega[j])=0$ for all $j>0$.  Dually, one can define the subcategory ${_{\omega}\X}$. Following \cite[Definition 2.4]{Wei2015}, an object in $\mathcal{G}_\omega={_{\omega}\X}\cap \X_\omega$ is called {\it $\omega$-Gorenstein}. Then $(\mathcal{G}_\omega, \mathcal{G}_{\omega})$ is an $\mathcal{\add \omega}$-mutation pair by \cite[Proposition 2.5(2)]{Wei2015}
. Moreover, if $\widehat{\X_\omega}=\T=\widecheck{_\omega\X}$, then $(\X_\omega, \add\omega,{_{\omega}\X})$ is an $\langle\add\omega\rangle$-localization triple  by the proof of \cite[Lemma 2.2 and Lemma 2.6(2)]{Wei2015}. Thus by Corollary \ref{cor:Verdier vs IY}, we have

 \begin{corollary}\label{wei's result}$($\cite[Theorem 2.7]{Wei2015}$)$ \ There is a triangle equivalence $\mathcal{G}_{\omega}/[\add\omega] \stackrel{\sim}\hookrightarrow\T/\langle\add\omega\rangle$.
 \end{corollary}


\end{document}